\documentclass{amsart}[12pt,article]

\usepackage{amsmath,amssymb,amscd,amsthm,indentfirst}
\usepackage[all]{xy}
\usepackage{enumerate}
\usepackage{graphicx}
\usepackage{multirow}

\newtheorem{thm}{Theorem}[section]
\newtheorem{prop}[thm]{Proposition}

\newtheorem{lem}[thm]{Lemma}
\newtheorem{conjecture}[thm]{Conjecture}

\theoremstyle{definition}
\newtheorem{defn}[thm]{Definition}
\newtheorem{remark}[thm]{Remark}
\newtheorem{notation}[thm]{Notation}

\def\A{\ensuremath{\mathcal{A}}}
\def\C{\ensuremath{\mathcal{C}}}
\def\D{\ensuremath{\mathcal{D}}}
\def\Q{\ensuremath{\mathcal{Q}}}

\def\Sym{\operatorname{Sym}}
\def\FF{\ensuremath{\mathbb{F}}}

\def\ZZ{\ensuremath{\mathbb{Z}}}

\def\aa{\mathbf a}
\def\cc{\mathbf c}
\def\ii{\mathbf i}
\def\jj{\mathbf j}
\def\kk{\mathbf k}

\def\xx{\mathbf x}
\def\dd{\boldsymbol{\delta}}
\def\ee{\boldsymbol{\epsilon}}

\def\rk{\operatorname{rk}}

\def\lcm{\operatorname{lcm}}

\def\Isom{\mathrm{Isom}}

\def\im{\operatorname{Im}}

\def\sgn{\operatorname{sgn}}
\def\diag{\operatorname{diag}}

\newcommand{\GL}{\operatorname{GL}\nolimits}
\newcommand{\SL}{\operatorname{SL}\nolimits}
\newcommand{\PSL}{\operatorname{PSL}\nolimits}
\newcommand{\GU}{\operatorname{GU}\nolimits}
\newcommand{\SU}{\operatorname{SU}\nolimits}
\newcommand{\PSU}{\operatorname{PSU}\nolimits}
\newcommand{\PGU}{\operatorname{PGU}\nolimits}
\newcommand{\PGL}{\operatorname{PGL}\nolimits}

\newcommand{\ls}[2]{{}^{{#1}\!}{#2}}

\newcommand{\qbox}[1]{\quad\hbox{#1}\quad}
\newcommand{\ovl}[1]{\overline{#1}}

\date{\today}
\title{A geometric approach to Quillen's conjecture}

\author{Antonio D\'{i}az Ramos}
\address{Departamento de {\'A}lgebra, Geometr{\'\i}a y Topolog{\'\i}a,
Universidad de M{\'a}\-la\-ga, Apdo correos 59, 29080 M{\'a}laga,
Spain.}
\thanks{First author supported by MEC grant MTM2016-78647-P and Junta de Andaluc\'ia grant FQM-213.} 
\email{adiazramos@uma.es}
\author{Nadia Mazza}
\address{Department of Mathematics and Statistics\\
Lancaster University\\
Lancaster, LA1 4YF\\
UK.}
\email{n.mazza@lancaster.ac.uk}

\begin{document}

\begin{abstract}
We introduce {\em admissible collections} for a 
finite group $G$ and use them to prove that most of the finite classical
groups in non-defining characteristic satisfy the
{\em Quillen dimension at $p$ property}, a strong version of Quillen's 
conjecture, at a given odd prime divisor $p$ of $|G|$. 
Compared to the methods in \cite{AS1993}, our techniques are 
simpler.

\end{abstract}

\subjclass[2010]{55U10; 05E45, 20J99.}
\maketitle

\smallskip

\section{Introduction}
\label{section:introduction}

Let $G$ be a finite group and let $p$ be a prime number dividing the
order of $G$. Let $\A_p(G)$ be the poset of all non-trivial elementary
abelian $p$-subgroups of $G$ ordered by inclusion and let $|\A_p(G)|$
be its realization as a topological space. This space is a simplicial
complex which has as $n$-simplices the 
chains of length $n$ in $\A_p(G)$: 
\[
P_0<P_1<\ldots<P_n\text{, with $P_i\in \A_p(G)$.}
\]
We denote by $O_p(G)$ be the largest normal $p$-subgroup of $G$, and by
$\rk_p(G)$ the $p$-rank of $G$, that is, the maximum of the ranks of the
elementary abelian $p$-subgroups of $G$. A finite elementary abelian
$p$-group $E$ has rank $r$ if $|E|=p^r$, or equivalently,
$r=\dim_{\FF_p}E$, when we regard $E$ as an $\FF_p$-vector space.

\begin{conjecture}[Quillen's conjecture \cite{Quillen1978}]
If $|\A_p(G)|$ is contractible then $O_p(G)\neq 1$.
\end{conjecture}

In \cite{AS1993}, the authors introduce the following property with rational coefficients. Here, we  work with integral coefficients and by $\widetilde H_*(|\A_p(G)|)$ we mean reduced integral simplicial homology (cf. \cite[Ch. 7]{rotman}), i.e., $\widetilde H_*(|\A_p(G)|)=H_*(|\A_p(G)|;\ZZ)$ for $*>0$ and $H_0(|\A_p(G)|;\ZZ)=\widetilde H_0(|\A_p(G)|)\oplus \ZZ$.

\begin{defn}[$\Q\D_p$]
\label{def:QDp}
The finite group $G$ with $r=\rk_p(G)$ has the {\em Quillen dimension at
  $p$} property, written $\Q\D_p$, if 
\[
\widetilde H_{r-1}(|\A_p(G)|)\neq 0.
\]
\end{defn}

Note that $r-1$ is the top dimension for which
$\widetilde H_*(|\A_p(G)|)$ can possibly be non-zero. Aschbacher and
Smith's approach to Quillen's conjecture \cite{AS1993} is the best
result so far.
In \cite[Theorem 3.1]{AS1993}, they consider $p$-extensions of finite
simple groups, that is, almost-simple groups with an
elementary abelian $p$-group inducing outer automorphisms. Aschbacher and
Smith prove that most of these $p$-extensions satisfy $\Q\D_p$, and
they list those which do not. Their main result asserts that if $p>5$ and $G$ does not have a unitary component $U_n(q)$ with
$q\equiv-1\pmod{p}$ and $q$ odd, then $G$ satisfies Quillen's
conjecture \cite[Main Theorem, p. 474]{AS1993}. 
  
\smallskip
In \cite{Diaz2016}, the first author proves that Quillen's conjecture
holds for solvable and $p$-solvable finite groups via new geometrical methods.
In the present paper, we elaborate on these methods with the objective
to find shorter and easier proofs of the results in \cite{AS1993}. We deal 
 with the alternating, symmetric and classical finite groups in non-defining
characteristic and for an odd prime $p$. By a finite {\em classical group}, we mean a finite
linear, unitary, symplectic or orthogonal group. With these methods, we prove the following result. 

\begin{thm}\label{thm:main}
Quillen's conjecture holds for the group $G$ at the odd prime $p$ in the following cases:
\begin{enumerate}
\item[(i)] $G$ is an alternating or symmetric group of degree $n$ for
  $p\geq 5$, and for $p=3$, $n=4,5,8$. 
\item[(ii)] $G$ is a linear, unitary, symplectic or orthogonal group
  defined over a field of characteristic different from $p$, unless $p|(q+1)$ and $G$ is a unitary group defined over $\mathbb F_{q^2}$ or $G$ is a
  symplectic or orthogonal group defined over $\mathbb F_2$ or $G=\PSL_3(q)$ with $q\equiv 1\pmod{3}$.
\end{enumerate}
\end{thm}

The structure of the paper is as follows.
In Section~\ref{section:faithfulcollections}, we define
{\em faithful collections} and discuss some of their properties. Such
a collection for a group $G$ consists of an arrangement of elementary
abelian $p$-subgroups of $G$ together with certain elements of $G$
that centralize/normalize these $p$-subgroups. In
Section~\ref{section:generalization}, we find further conditions on
faithful collections which imply that a finite group has the Quillen
dimension at $p$ property $\Q\D_p$, and we call such faithful
collections {\em admissible}.
In Section ~\ref{section:pextensionsB1}, we briefly study when the
Quillen dimension at $p$ property of a given finite group is inherited
by its quotient groups.
Then we show the Quillen dimension at $p$ property for the symmetric
and alternating groups, and for the finite classical groups in
non-defining characteristic, excluding certain cases when $p=3$.
Thus Theorem \ref{thm:main} summarizes these results.
We also present some limitations of our method and open questions.

\section{Faithful collections}
\label{section:faithfulcollections}

In this section, we define faithful collections for a finite
group $G$. Given $E\in\A_p(G)$, we regard $E$ as an $\FF_p$-vector
space (generally written additively) and a faithful collection for
$G$, if it exists, is a certain arrangement of hyperplanes of $E$ and
elements of $G$ subject to certain constraints.

\begin{defn}\label{def:sequences}
Let $r$ be a positive integer. For any integer $l$ with
$0\leq l\leq r$, we define an \emph{$l$-tuple for $r$} to be
an ordered sequence of elements $\ii=[i_1,\ldots,i_l]$ with $1\leq i_j\leq r$
and no repetition. By $S^r_l$ we denote the set of all
$l$-tuples for $r$ for a given $l$, and by
$S^r=\cup_{l=0,\ldots,r}S^r_l$ the set of all $l$-tuples
for $r$ for all $0\leq l\leq r$. 
\end{defn}

\begin{defn}\label{def:sequencesubspace}
Let $E=\langle e_1,\dots,e_r \rangle$ be an elementary abelian $p$-subgroup of
$G$ of rank $r$. For each $l$-tuple $\ii=[i_1,\ldots,i_l]\in S^r_l$, set:
\[
E_\ii=E_{[i_1,\ldots,i_l]}=
\langle e_1,\ldots,\widehat{e_{i_1}},\ldots,\widehat{e_{i_l}},\ldots,e_r\rangle,
\]
for the subgroup of $E$ generated by all the $e_i$, except
$e_{i_1},\dots,e_{i_l}$.
\end{defn}

We abuse notation and write
$E_i=E_{[i]}=\langle e_1,\ldots,\widehat{e_i},\ldots,e_r\rangle$. In
Definition~\ref{def:sequencesubspace}, the subgroups $E_i$ are
subspaces of $E$ of codimension $1$. Note
that we have the following: 
\[
E_{[i_1,\ldots,i_l]}=E_{i_1}\cap\ldots\cap E_{i_l}\qbox,
E_\emptyset=E=\langle E_i,e_i\rangle
\qbox, E_{[1,2,\ldots,r]}=\{0\},
\]
and, in particular:
\begin{equation}\label{equ:eideterminedbyhyperplanesuptopower}
\langle e_i\rangle =E_{[1,\ldots,\hat i,\ldots,r]}=E_1\cap \ldots \cap\widehat{E_i}\cap\ldots \cap E_r.
\end{equation} 

\begin{remark}\label{remark:equalsequencesubspaces}
Note that 
\[
E_{[i_1,\ldots,i_l]}=E_{[i'_1,\ldots,i'_{l'}]}\Leftrightarrow\;l=l'\qbox{and}
\{i_1,\ldots,i_l\}=\{i'_1,\ldots,i'_{l'}\}. 
\]
\end{remark}

\begin{defn}[Faithful collection]
\label{def:failthfulcollection}
Let $E=\langle e_1,\dots,e_r\rangle$ be an elementary abelian
$p$-subgroup of $G$ of rank $r$. For $\cc=(c_1,\ldots,c_r)\in G^r$, $\ii\in S^r$ and $\ee=(\epsilon_1,\ldots,\epsilon_r)\in\{-1,0,1\}^r$ set  
\[
\ls{\cc^{\ee}}E_\ii=\ls{c_1^{\epsilon_1}\ldots c_r^{\epsilon_r}}E_\ii.
\]
We say that the collection $\{E_i,c_j\}$ is \emph{faithful} if for any
such $\ee$ and $\ii$ we have  
\[
\ls{\cc^{\ee}}E_\ii\leq E \Rightarrow \ls{\cc^{\ee}}E_\ii=E_\ii.
\]
\end{defn}
So, a collection is faithful if the subspaces $E_\ii$ of $E$ in Definition
\ref{def:sequencesubspace} can be conjugated by the elements
$\cc^{\ee}$ of $G$ to a subspace of $E$ if and only if $\cc^{\ee}$ normalizes $E_\ii$. This mimics a
property of coprime actions, see \cite[Proposition 2.2(1)]{Diaz2016}.
Note that if $\ii\in S_0^r$ or $S_r^r$, the implication  
\[
\ls{\cc^{\ee}}E_\ii\leq E \Rightarrow \ls{\cc^{\ee}}E_\ii=E_\ii.
\]
is vacuous. 

The next result is a characterization of faithful collections in terms
of the elements $e_i$ generating $E$ in
Definition~\ref{def:sequencesubspace}. By Equation
\eqref{equ:eideterminedbyhyperplanesuptopower}, these elements are
determined by the hyperplanes $E_i$ up to a power. 

\begin{lem}\label{lem:characterizationfaithfulcollection}
Let $E=\langle e_1,\dots,e_r\rangle$ be an elementary abelian
$p$-subgroup of $G$ of rank $r$. Let $\cc=(c_1,\ldots, c_r)\in G^r$. Then the collection $\{E_i,c_j\}$ is faithful if and only if for any $\ee\in\{-1,0,1\}^r $ and any $1\leq i\leq r$ we have 
\[
\ls{\cc^{\ee}}\langle e_i\rangle\leq E \Rightarrow
\ls{\cc^{\ee}}\langle e_i\rangle=\langle e_i\rangle. 
\]
\end{lem}
\begin{proof}
That a faithful collection satisfies the given conditions is clear by
considering the $r-1$-tuples $[1,\ldots,\hat i,\ldots,r]$.
Conversely, let $\ii\in S_{r-l}^r$ for some $0\leq l\leq r$,
and let $\ee=(\epsilon_1,\dots,\epsilon_r)\in\{-1,0,1\}^r$ such that 
\[
\ls{\cc^{\ee}}E_\ii\leq E,
\]
where $E_\ii=\langle e_{j_1},\ldots,e_{j_l}\rangle$ has rank $l$.
In particular, we have 
\[
\ls{\cc^{\ee}}\langle e_{j_k} \rangle \leq E
\]
for $1\leq k\leq l$, and so, by hypothesis, $\ls{\cc^{\ee}}\langle
e_{j_k} \rangle=\langle e_{j_k}\rangle$ for $1\leq k\leq l$.
It follows that $\ls{\cc^{\ee}}E_\ii=
\ls{\cc^{\ee}}\langle e_{j_1},\ldots,e_{j_l}\rangle=\langle
\ls{\cc^{\ee}}e_{j_1},\ldots,\ls{\cc^{\ee}}e_{j_l}\rangle=
\langle e_{j_1},\ldots,e_{j_l}\rangle=E_\ii$. 
\end{proof}

The next result gives sufficient conditions for the existence of a
faithful collection. 

\begin{lem}\label{lem:sufficientforfaithfulcollection}
Let $E=\langle e_1,\dots,e_r\rangle$ be an elementary abelian
$p$-subgroup of $G$ of rank $r$ and let $\cc=(c_1,\ldots, c_r)\in G^r$.
Suppose $c_i\in C_G(E_i)\setminus C_G(e_i)$ and
$[c_i,c_j]=1$ for all $i,j$. Then the collection $\{E_i,c_j\}$ is faithful.
\end{lem}

\begin{proof}
We use Lemma \ref{lem:characterizationfaithfulcollection}. Consider
 $\ee=(\epsilon_1,\ldots,\epsilon_r)$ with $\epsilon_i\in
\{-1,0,1\}$ and let $1\leq i_0\leq r$ such that 
\[
\ls{\cc^{\ee}}\langle e_{i_0}\rangle\leq E.
\]
Then we have
\[
\ls{\cc^{\ee}}e_{i_0}=
\ls{c_1^{\epsilon_1}\ldots c_r^{\epsilon_r}}e_{i_0}=
\ls{c_{i_0}^{\epsilon_{i_0}}}e_{i_0}=\lambda_1
e_1+\ldots +\lambda_{i_0} e_{i_0}+\ldots +\lambda_r e_r, 
\]
where $\lambda_i\in \FF_p$ (not all zero), and in the second equality
we have used that $c_i\in C_G(E_i)$, that $[c_i,c_{i_0}]=1$
and that $e_{i_0}\in E_i$ for $i\neq i_0$.
Assume $\lambda_{i_1}\neq 0$ for some $i_1\neq i_0$. Then,
\[
E=\langle E_{i_1},\ls{c_{i_0}^{\epsilon_{i_0}}}e_{i_0}\rangle .
\]
Now, by hypothesis $c_{i_1}$ centralizes $e_{i_0}\in E_{i_1}$, and we
calculate 
\[
\ls{c_{i_1}c_{i_0}^{\epsilon_{i_0}}}e_{i_0}=
\ls{c_{i_0}^{\epsilon_{i_0}}c_{i_1}}e_{i_0}=\ls{c_{i_0}^{\epsilon_{i_0}}}e_{i_0}, 
\]
using that $[c_{i_1},c_{i_0}]=1$, that $c_{i_1}\in
C_G(E_{i_1})$ and that $e_{i_0}\in E_{i_1}$ (as $i_0\neq i_1$). This is
a contradiction with $c_{i_1}\notin C_G(E)$, and hence $\lambda_i=0$ for
all $i\neq i_0$ and $\ls{\cc^{\ee}} e_{i_0}=\lambda_{i_0}
e_{i_0}\in\langle e_{i_0}\rangle$. 
\end{proof}

In the next section, we will use the existence of faithful collections
with $E$ of rank $r=\rk_p(G)$ subject to certain constraints to prove
that $\Q\D_p$ holds for $G$, see Theorem
\ref{theorem:generalizationcentralizernormalizer}. 

By constrast, we prove that, under certain assumptions on $G$, there
cannot be a faithful collection subject to the conditions in Lemma
\ref{lem:sufficientforfaithfulcollection}.
Recall that a {\em maximal elementary abelian $p$-subgroup of $G$} is
a maximal element in the poset $\A_p(G)$, i.e., an elementary
abelian $p$-subgroup $E$ of $G$ which is not properly contained in any
other elementary abelian $p$-subgroup of $G$. But $E$ need not have
maximal rank $\rk_p(G)$. 

\begin{lem}\label{lem:pstable}
Suppose that $O_p(G)\cap Z(G)>1$ for some odd prime
$p$. Let $E$ be a maximal elementary abelian $p$-subgroup of $G$, not
necessarily of maximal order. 
Then there exists no collection $\{E_i,c_j\}$
for $E$ subject to $c_i\in C_G(E_i)\setminus C_G(e_i)$
for all $i$. 
\end{lem}

\begin{proof}
Let $E=\langle e_1,\dots,e_r\rangle\in\A_p(P)$ be a maximal element,
and suppose that $\{E_i,c_j\}$ is a collection subject to
$c_i\in C_G(E_i)\setminus C_G(e_i)$ for all $i$.
Let $V=\Omega_1\big(O_p(G)\cap Z(G)\big)$. Note that $E=VE\geq V$ by
maximality of $E$, and the fact that $V$ lies in the center of every
Sylow $p$-subgroup of $G$. 

Let $1\neq v=\sum_{i=1}^r\lambda_ie_i\in V$, and without loss, suppose
$\lambda_1=1$. (Here we use the additive notation of $E$ seen as
vector space.)
So $v=e_1+v'$ with $v'\in E_1$. Because $v\in V\leq Z(G)$ and
$c_1\in C_G(E_1)$, we have
$v=\ls{c_1}v=\ls{c_1}(e_1+v')=(\ls{c_1}e_1)+v'$. Therefore
$e_1=v-v'=\ls{c_1}e_1$, saying that $c_1\in C_G(e_1)$, a
contradiction. 
\end{proof}


\section{Generalization}\label{section:generalization}

We want to find sufficient conditions which imply that $G$ has the
Quillen dimension at $p$ property, or more generally, which imply that
Quillen's conjecture holds for $G$. Theorem
\ref{theorem:generalizationcentralizernormalizer} below is a generalization of
\cite[Theorem 5.3]{Diaz2016}. 

\begin{defn}\label{def:sequenceandsignature}
For an $l$-tuple $\ii=[i_1,\ldots,i_l]\in S^r_l$, we define
the {\em signature}
\[
\sgn(\ii)=(-1)^{n+m}\qbox{of $\ii$, where}
\]
\begin{itemize}\item $n$ is the number of transpositions
we need to apply to the $l$-tuple $\ii$ to rearrange it in increasing order
$[j_1,\ldots,j_l]$, and\item $m$ is the number of positions in which
$[j_1,\ldots,j_l]$ differ from $[1,\ldots,l]$.\end{itemize}
\end{defn}
Note that in Definition~\ref{def:sequenceandsignature}, the number $n$
of transpositions is not uniquely defined, but its parity is.

For instance, if $\ii=[1,4,2]$, then $n=1$, since we need to apply $(2,4)$
to reorder $\ii$ as $[1,2,4]$, and $m=1$, since $[1,2,4]$ differs from
$[1,2,3]$ only in one place. Thus $\sgn(\ii)=1$.

Let $E=\langle e_1,\dots,e_r\rangle$ be an elementary
abelian $p$-subgroup of rank $r$ of the group $G$. We now generalize the chains introduced in
\cite[Section 3]{Diaz2016} to the case when $E$ need not have
maximal rank $\rk_p(G)$ and $G$ need not be a semi-direct product. We consider the poset $\A_p(E)$ and its order complex $\Delta(\A_p(E))$. We define an element of the integral simplicial chains of dimension $r-1$, $C_{r-1}(\Delta(\A_p(E)))$.
For $\ii=[i_1,\ldots,i_{r-1}]\in S^r_{r-1}$, we define the
$(r-1)$-simplex 
$$\sigma_{\ii}=\big(E_{[i_1,\ldots,i_{r-1}]}<E_{[i_1,\dots,i_{r-2}]}<
\dots<E_{i_1}<E\big)\in\Delta(\A_p(E)),$$
and for $a\in \ZZ$, the chain
$$Z_{E,a}=a\sum_{\ii\in S^r_{r-1}}\sgn(\ii)\sigma_{\ii}
\qbox{in}C_{r-1}(\Delta(\A_p(E))).$$
By definition of the differential,
$$d(Z_{E,a})=a\sum_{0\leq j<r}\sum_{\ii\in S^r_{r-1}}(-1)^j\sgn(\ii)d_j(\sigma_{\ii})
\qbox{where}$$
$$d_j(\sigma_{\ii})=\big(E_{[i_1,\ldots,i_{r-1}]}<\dots<E_{[i_1,\dots,i_{r-j}]}<
E_{[i_1,\dots,i_{r-j-2}]}<\dots<E_{i_1}<E\big)$$
removes the $(j+1)$-st term from the left in the chain for $0\leq j<r$. So, in particular,
$$d_0(\sigma_{\ii})=\big(E_{[i_1,\dots,i_{r-2}]}<\dots<E\big)
\qbox{and}
d_{r-1}(\sigma_{\ii})=\big(E_{[i_1,\ldots,i_{r-1}]}<\dots<E_{i_1}\big).$$
Hence, define the $(r-2)$-simplex
$$\tau_{\ii}=d_{r-1}(\sigma_{\ii})=\big(E_{[i_1,\ldots,i_{r-1}]}<\dots<E_{i_1}\big).$$
Let us recall the following useful property
(\cite[Proposition 3.2]{Diaz2016}).
\begin{prop}\label{prop:dz}
With the above notation,
\[
d(Z_{E,a})=(-1)^{r-1}a\sum_{\ii\in S^r_{r-1}}\sgn(\ii)\tau_{\ii}.
\]
\end{prop}

\begin{proof}  
Suppose that $d_k(\sigma_{\ii})=d_l(\sigma_{\jj})$ for some
$\ii,\jj\in S^r_{r-1}$ and $0\leq k,l\leq r-1$, that is,
\begin{multline*}
\big(E_{[i_1,\ldots,i_{r-1}]}<\dots<E_{[i_1,\dots,i_{r-k}]}<
E_{[i_1,\dots,i_{r-k-2}]}<\dots<E_{i_1}<E\big)=\\
=\big(E_{[j_1,\ldots,j_{r-1}]}<\dots<E_{[j_1,\dots,i_{r-l}]}<
E_{[j_1,\dots,j_{r-l-2}]}<\dots<E_{j_1}<E\big).
\end{multline*}
For both chains to be equal, the jump by an index $p^2$ must occur in
the same place, saying that $k=l$. If $0<k=l<r-1$, then, by Remark
\ref{remark:equalsequencesubspaces}, the tuples $\ii$ and $\jj$ are
identical but for $\{ i_{r-k-1}, i_{r-k} \}=\{ j_{r-k-1}, j_{r-k}\}$.
So either $\ii=\jj$, or
$\ii$ and $\jj$ differ by one transposition and hence
$\sgn(\ii)=-\sgn(\jj)$. In the latter case, the
corresponding summands $\sgn(\ii)d_k(\sigma_\ii)$ and $\sgn(\jj)
d_l(\sigma_\jj)$ add up to zero. Assume now that $k=l=0$. Then, again
by Remark \ref{remark:equalsequencesubspaces},
$[j_1,\dots,j_{r-2}]=[i_1,\dots,i_{r-2}]$ and either $\ii=\jj$ or
$j_{r-1}\neq i_{r-1}$. In the latter case, by \cite[Lemma
  2.4]{Diaz2016}, $\sgn(\ii)=-\sgn(\jj)$ and again the two
terms cancel each other out. Finally, if $k=l=r-1$, the tuples
$\ii$ and $\jj$ are identical and the terms contribute to the sum in
the statement of the lemma. 
\end{proof}

\smallskip  
Consider the order complex $\Delta(\A_p(G))$. Then the group $G$ acts by conjugation on $C_*(\Delta(\A_p(G)))$.
For $x\in G$, the element
\[
\ls xZ_{E,a}=a\sum_{\ii\in S^r_{r-1}}\sgn(\ii)
\ls x{\sigma_{\ii}}=
a\sum_{\ii\in S^r_{r-1}}\sgn(\ii)
\big(\ls xE_{[i_1,\ldots,i_{r-1}]}<\dots<\ls xE_{i_1}<\ls xE\big)
\]
belongs to $\Delta(\A_p(\ls xE))\subseteq \Delta(\A_p(G))$. 

Let $J$ be a non-empty finite indexing set and consider subsets $\xx=\{x_j\}_{j\in J}\subseteq G$ and 
$\aa=\{a_j\}_{j\in J}\subseteq \ZZ$. Define the chain in $C_{r-1}(\Delta(\A_p(G)))$:
\begin{equation}\label{eq:z}
Z_{G,\xx,\aa}=\sum_{j\in J}\ls{x_j}Z_{E,a_j}=
\sum_{j\in J}a_j\sum_{\ii\in S^r_{r-1}}
\sgn(\ii)\ls{x_j}{\sigma_{\ii}}.
\end{equation}
Then, by Proposition \ref{prop:dz},
\[
d(Z_{G,\xx,\aa})=(-1)^{r-1}\sum_{j\in J}a_j\sum_{\ii\in S^r_{r-1}}
\sgn(\ii)\ls{x_j}{\tau_{\ii}}.
\]
So, given $j\in J$ and $\ii\in S^r_{r-1}$, the coefficients of
$\ls{x_j}{\sigma_{\ii}}$ and $\ls{x_j}{\tau_{\ii}}$ are respectively:
\begin{align}
C_{j,\ii}&:=\sum_{(l,\kk)\in\C(j,\ii)}a_l\sgn(\kk)
\qbox{and}\label{equ:coefficientforZgeneral}\\ 
D_{j,\ii}&:=(-1)^{r-1}\sum_{(l,\kk)\in\D(j,\ii)}a_l\sgn(\kk),
\label{equ:coefficientfordZgeneral}
\end{align}
where
\begin{align*}
\C(j,\ii)&=\{(l,\kk)\in J\times S^r_{r-1}~|~\ls{x_l}E_{[k_1,\dots,k_t]}=
\ls{x_j}E_{[i_1,\dots,i_t]}~,~\forall~0\leq t<r\}\qbox{and}
\label{equ:coefficientforZgeneral}\\ 
\D(j,\ii)&=\{(l,\kk)\in J\times S^r_{r-1}~|~\ls{x_l}E_{[k_1,\dots,k_t]}=
\ls{x_j}E_{[i_1,\dots,i_t]}~,~\forall~1\leq t<r\}.
\end{align*}
Note that $\C(j,\ii)\subseteq\D(j,\ii)$ (and recall that
$E_\emptyset=E$).

If we further assume that $E$ is a \emph{maximal} elementary abelian
$p$-subgroup of $G$, we want to find sufficient conditions for the
existence of a non-zero cycle in $\widetilde H_{r-1}(|\A_p(G)|)$.

\begin{thm}\label{thm:noncontractiblemoregeneral}
Let $E=\langle e_1,\dots,e_r\rangle$ be a \emph{maximal} elementary
abelian $p$-subgroup of rank $r$ of the group $G$. If the subsets $\xx=\{x_j\}_{j\in J}\subseteq G$
and $\aa=\{a_j\}_{j\in J}\subseteq \ZZ$ satisfy that 
$C_{j,\ii}\neq 0$ for some $j\in J$ and some
$\ii\in S^r_{r-1}$  and that
 $D_{j,\ii}=0$ for all  $j\in J$ and all
$\ii\in S^r_{r-1}$, then
\[
0\neq [Z_{G,\xx,\aa}]\in \widetilde H_{r-1}(|\A_p(G)|).
\]
In particular, $|\A_p(G)|$ is not contractible and
Quillen's conjecture holds for $G$.
If furthermore $r=\rk_p(G)$ then $\Q\D_p$ holds for $G$. 
\end{thm}

\begin{proof}
Consider the chain $Z_{G,\xx,\aa}\in C_{r-1}(\Delta(\A_p(G)))$
defined in Equation \eqref{eq:z}. The condition in the statement for
the coefficients $C_{j,\ii}$ is clearly equivalent to
$Z_{G,\xx,\aa}\neq0$. The condition on the coefficients $D_{j,\ii}$ is
equivalent to $d(Z_{G,\xx,\aa})=0$ (cf. also \cite[Proposition
  4.2]{Diaz2016}). By the maximality of $E$ , this cycle cannot be a
boundary and hence it gives rise to a non-zero homology class in
$\widetilde H_{r-1}(|\A_p(G)|)$. 
\end{proof}

The assumptions of Theorem~\ref{thm:noncontractiblemoregeneral} are
fulfilled when we can find a faithful collection subject to certain
constraints, as described in the next theorem. 

\begin{thm}\label{theorem:generalizationfaithfulcollection}
Let $E=\langle e_1,\dots,e_r\rangle$ be a maximal elementary abelian
$p$-subgroup of $G$ of rank $r$ and let $\{E_i,c_j\}$ be
a faithful collection such that $[c_i,c_j]=1$ for all $i,j$. Set
$J=\{\dd=(\delta_1,\ldots,\delta_r)\;,\;\delta_i\in\{0,1\}\}$, let
$\aa=\{a_{\dd}\}_{\dd\in J}\subseteq \ZZ$ and consider the following subset of
$G$: 
\[
\xx=\{\cc^{\dd}=c_1^{\delta_1}\cdots
c_r^{\delta_r}|\dd\in J\}.
\]
Then, for each such $\dd\in J$ and each $\ii=[i_1,\ldots,i_{r-1}]\in S^r_{r-1}$: 
\begin{enumerate}
\item \label{theorem:generalizationfaithfulcollectionCsimplified} 
$\displaystyle C_{\dd,\ii}=
\sgn(\ii)\big(\sum_{\dd'} a_{\dd'}\big)$, where the sum is over all
$\dd'\in J$ such that $\cc^{\dd-\dd'}\in N_G(E)$.
\item \label{theorem:generalizationfaithfulcollectionDsimplified} 
$\displaystyle D_{\dd,\ii}=(-1)^{r-1}\sgn(\ii)\big(\sum_{\dd'} a_{\dd'}\big)$,
where the sum is over all $\dd'\in J$ such that $\cc^{\dd-\dd'}\in N_G(E_{i_1})$.
\item \label{theorem:generalizationfaithfulcollectionCsimplifiedagain}
If $c_i\in C_G(E_i)\setminus N_G(\langle e_i\rangle)$ for all $1\leq i\leq r$, then
$\displaystyle
C_{\dd,\ii}=\sgn(\ii) a_{\dd}$.
\item \label{theorem:generalizationfaithfulcollectionDsimplifiedagain}
If $c_i\in N_G(E_i)$ for all $1\leq i\leq r$, then
$$D_{\dd,\ii}=(-1)^{r-1}\sgn(\ii)\big(\sum (a_{\dd'}+a_{\dd''})\big),$$
where the sum runs through  the pairs $\dd',\dd''\in J$ such that $\cc^{\dd-\dd'},\cc^{\dd-\dd''}\in N_G(E_{i_1})$, $\delta''_j=\delta'_j$ for all $j\neq i_1$ and $\delta'_{i_1}+~\delta''_{i_1}=1$.

\end{enumerate}
\end{thm}

\begin{proof}
Using the faithfulness condition in Definition
\ref{def:failthfulcollection} and Remark
\ref{remark:equalsequencesubspaces}, a straightforward computation
shows that $(\dd',\kk)\in\C(\dd,\ii)$ (cf. definition of $\C(\dd,\ii)$
in Equation \eqref{equ:coefficientforZgeneral} above) if and only if
$\kk=\ii$ and $\cc^{\dd-\dd'}\in N_G(E)$, and similarly,
$(\dd',\kk)\in\D(\dd,\ii)$ (cf. definition of $\D(\dd,\ii)$ in
Equation \eqref{equ:coefficientfordZgeneral} above) if and only if
$\kk=\ii$ and 
$\cc^{\dd-\dd'}\in N_G(E_{i_1})$. In other words,
\begin{align*}
\C(\dd,\ii)&=\{(\dd',\ii)\in J\times S^r_{r-1}~|~\cc^{\dd-\dd'}\in N_G(E)\}\qbox{and}
\label{equ:coefficientforZgeneral}\\ 
\D(\dd,\ii)&=\{(\dd',\ii)\in J\times S^r_{r-1}~|~\cc^{\dd-\dd'}\in N_G(E_{i_1})\}.
\end{align*}
From here, we immediately get the assertions \eqref{theorem:generalizationfaithfulcollectionCsimplified} and
\eqref{theorem:generalizationfaithfulcollectionDsimplified}.

Suppose the hypotheses in point $(3)$ hold. Again, since
$\ls{\cc^{\dd}}E=\ls{\cc^{\dd'}}E
\Longleftrightarrow \ls{\cc^{\dd-\dd'}}E=E$,
where each $\dd_j-\dd'_j\in\{-1,0,+1\}$, the faithfulness condition
implies that  
\[
\ls{c^{\dd-\dd'}}\langle e_i\rangle=
\ls{c_i^{\dd_i-\dd'_i}}\langle e_i\rangle=\langle e_i\rangle,
\]
where we have used that $[c_i,c_j]=1$, that $c_j\in C_G(E_j)$ and that
$e_i\in E_j$ for $i\neq j$.
By assumption $c_i\notin N_G(\langle e_i\rangle)$, which forces
$\delta_i=\delta'_i$, and this holds for all $1\leq i\leq r$.

Finally, for point
\eqref{theorem:generalizationfaithfulcollectionDsimplifiedagain},
let $\dd'$ be such that
$\ls{\cc^{\dd}}E_{i_1}=\ls{\cc^{\dd'}}E_{i_1}$.
If $\delta'_{i_1}=0$ then
$\dd'=(\delta'_1,\ldots,0,\ldots,\delta'_r)$ 
and we conjugate by $c_{i_1}$ to obtain
\[
\ls{c_{i_1}\cc^{\dd'}}E_{i_1}=\ls{\cc^{\dd'}c_{i_1}}E_{i_1}=
\ls{\cc^{\dd'}}E_{i_1}=\ls{\cc^{\dd}}E_{i_1},
\]
where we have used that $[c_i,c_j]=1$ and that $c_{i_1}\in N_G(E_{i_1})$. So we
deduce that $\dd''=(\delta'_1,\ldots,1,\ldots,\delta'_r)$ also
appears in the sum in
\eqref{theorem:generalizationfaithfulcollectionDsimplified} of this theorem.
If $\delta'_{i_1}=1$ then $\dd'=(\delta'_1,\ldots,1,\ldots,\delta'_r)$
and we conjugate by $c_{i_1}^{-1}$ to obtain 
\[
\ls{c_{i_1}^{-1}\cc^{\dd'}}E_{i_1}=\ls{\cc^{\dd'}c_{i_1}^{-1}}E_{i_1}=
\ls{\cc^{\dd'}}E_{i_1}=\ls{\cc^{\dd}}E_{i_1},
\]
where we have used that $[c_i,c_j]=1$ and that $c_{i_1}\in N_G(E_{i_1})$. So we
deduce that $\dd''=(\delta'_1,\ldots,0,\ldots,\delta'_r)$ also
appears in the sum. 
\end{proof}

\begin{thm}\label{theorem:generalizationcentralizernormalizer}
Let $E=\langle e_1,\dots,e_r\rangle$ be a maximal elementary abelian
$p$-subgroup of $G$ of rank $r$  and assume that
there are elements $c_i\in C_G(E_i)\setminus N_G(\langle e_i\rangle)$ with $[c_i,c_j]=1$
for all $1\leq i,j\leq r$.
Then $\widetilde H_{r-1}(|\A_p(G)|)\neq 0$ and hence Quillen's
conjecture holds for $G$. If furthermore $r=\rk_p(G)$, then $\Q\D_p$
holds for $G$.
\end{thm}

\begin{proof}
Note that the collection $\{E_i,c_j\}$ is faithful by Lemma
\ref{lem:sufficientforfaithfulcollection}. Now apply Theorem
\ref{thm:noncontractiblemoregeneral} with $A=\ZZ$, and Theorem \ref{theorem:generalizationfaithfulcollection}\eqref{theorem:generalizationfaithfulcollectionCsimplifiedagain} and
\eqref{theorem:generalizationfaithfulcollectionDsimplifiedagain} with
$a_{\dd}=(-1)^{\dd}=(-1)^{\delta_1+\ldots+\delta_r}$ for
$\dd=(\delta_1,\ldots,\delta_r)$, $\delta_i\in\{0,1\}$. 
\end{proof}

\begin{defn}\label{def:ad-coll}
A collection satisfying the assumptions of
Theorem~\ref{theorem:generalizationcentralizernormalizer} is called
{\em admissible}. That is, given a maximal elementary abelian
$p$-subgroup $E=\langle e_1,\dots,e_r\rangle$ of $G$ with
$\rk_p(E)=r$, an admissible collection for $G$ is a collection $\{E_i,c_j\}$
of subgroups $E_i=\langle e_1,\ldots,\widehat{e_i},\ldots,e_r\rangle$
of $G$ and elements $c_j$ of $G$, such that
$c_i\in C_G(E_i)\setminus N_G(\langle e_i\rangle)$ and $[c_i,c_j]=1$
for all $1\leq i,j\leq r$. Note that such a collection is faithful by
Lemma \ref{lem:sufficientforfaithfulcollection}.
\end{defn}


\section{Applications to symmetric and classical groups}
\label{section:pextensionsB1}

We start this section with an observation which is useful when
investigating Quillen's conjecture and $\Q\D_p$ for $p'$-central
quotients of finite groups.

\begin{remark}\label{rem:faithfultoquotient}
Let $E=\langle e_1,\dots,e_r\rangle$ be a maximal elementary abelian
$p$-subgroup of $G$ of rank $r\leq \rk_p(G)$ and let
$c_1,\ldots,c_r\in G$. Let $N$ be a normal $p'$-subgroup of $G$
satisfying that for each $\ii\in S^r_l$ and for all
$\ee=(\epsilon_1,\dots,\epsilon_r)$ with $\epsilon_i\in\{-1,0,1\}$ we
have $\ls{\cc^{\ee}}E_{\ii}\cap EN\leq E$. In particular this holds
whenever $N$ is a normal $p'$-subgroup of $G$ which also centralizes
$E$ (for instance, if $N$ is a central $p'$-subgroup of $G$). Then: 
\begin{enumerate}
\item\label{it:quotient1} $\{E_i,c_j\}$ is faithful if and only if
$\{\ovl{E_i},\ovl{c_j}\}$ is faithful. 
\item\label{it:quotient2} $c_{i}\in C_G(E_i)\setminus N_G(\langle e_i\rangle)$
if and only if
$\ovl{c_{i}}\in C_{\ovl G}(\ovl E_i)\setminus N_{\ovl G}(\langle \ovl{e_i}\rangle)$. 
\end{enumerate}
\end{remark}

\begin{proof}
For part \eqref{it:quotient1}, assume first that $\{E_i,c_j\}$ is
faithful and that
$\ls{\ovl{\cc}^{\ee}}{\ovl e_i}\leq\ovl{E}$, i.e.,
that $\ls{\cc^{\ee}}e_i\in EN$.
Then, by assumption we have $\ls{\cc^{\ee}}e_i\in E$ and, as
$\{E_i,c_j\}$ is faithful, we get that
$\cc^{\ee}\in N_G(\langle e_i\rangle)$. Hence we also get
$\overline{\cc}^{\ee}\in N_{\ovl G}(\langle \ovl e_i\rangle)$. 
Conversely, suppose that $\{\ovl{E_i},\ovl{c_j}\}$ is faithful.
Suppose that $\ls{\cc^{\ee}}e_i\in E$.
Then $\ls{\ovl{\cc}^{\ee}}{\ovl e_i}\in\ovl E$ and by assumption 
$\ovl{\cc}^{\ee}\in N_{\ovl G}(\langle\ovl e_i\rangle)$. It follows
that $\ls{\cc^{\ee}}e_i\in \langle e_i\rangle N\leq E$. By
Dedekind's modular law,
$\langle e_i\rangle N\cap E=\langle e_i\rangle(N\cap E)=
\langle e_i\rangle$. Thus $\{E_i,c_j\}$ is faithful.

For part~\eqref{it:quotient2}, assume that
$c_i\in C_G(E_i)\setminus N_G(\langle e_i\rangle)$ for all $i$.
Then $\ovl{c_i}\in C_{\ovl G}(\ovl{E_i})$. We claim that
$\ovl{c_i}\notin N_{\ovl G}(\langle\ovl{e_i}\rangle)$. Indeed, by
assumption, if $\ls{c_i}e_i=e_i^{a_i}n$ for some $n\in N$ and integer
$0<a_i<p$, then
$\ls{c_i}e_i\in\ls{c_i}{\langle e_i\rangle}\cap EN\leq E$
and since $E\cap N=\{1\}$, we must have $n=1$ and
$c_i\in N_G(\langle e_i\rangle)$.

Conversely, if $\ovl{c_i}\notin N_{\ovl G}(\langle\ovl{e_i}\rangle)$,
then $c_i$ cannot normalize $\langle e_i\rangle$ for any
$1\leq i\leq r$. 
It remains to see that if $\ovl{c_{i}}\in C_{\ovl G}(\ovl E_i)$, then
$c_i$ centralizes $E_i$. For any $j\neq i$, if
$\ls{c_i}e_j\in e_jN$, then by assumption,
$\ls{c_i}e_j\in\ls{c_i}{\langle e_j\rangle}\cap EN\leq E$, and so
we must have $\ls{c_i}e_j=e_j$ since $E\cap N=\{1\}$.
\end{proof}

In the remainder of this section, we exhibit admissible collections
for many alternating groups and finite classical groups of Lie type in
non-defining characteristic, thereby proving that these groups possess
the Quillen dimension at $p$ property (cf. Definition \ref{def:ad-coll}).

\smallskip
We write $\Sym_X$ for the full permutation group on a finite set
$X$. If $X=\{1,\dots,n\}$, then we write $\Sym_n$
instead. Accordingly, we write $A_X$ and $A_n$ for the corresponding
alternating groups.

\begin{thm}\label{thm:an}
Suppose that $p>3$ and let $n\geq p$.
Let $G=A_n$ or $G=\Sym_n$.
Then there exists an admissible collection for $G$.
\end{thm}

\begin{proof}
We refer the reader to \cite[Ch. IV]{AM} for the $p$-local structure of
the alternating and symmetric groups.
The $p$-rank of $G$ is $r=\rk_p(A_n)=\lfloor \frac{n}{p} \rfloor$.
Consider the maximal elementary abelian $p$-subgroup $E$ of rank $r$
generated by  
\[
e_1=(1,\dots, p),\dots,~e_i=((i-1)p+1,\dots, ip),
\dots,~e_r=(r(p-1)+1,\dots, rp).
\]
Define hyperplanes of $E$ by 
$E_i=\langle e_1,e_2,\ldots,\hat e_i,\ldots, e_r\rangle$. 
Write $n=rp+b$ with $0\leq b<p$.
Recall that $|N_{\Sym_n}(\langle e_i\rangle):N_{A_n}(\langle e_i\rangle)|=
|C_{\Sym_n}(E_i):C_{A_n}(E_i)|=2$, and that,
since $p>3$, the groups $A_n$ are simple for $n\geq p$; in
particular, its Sylow $p$-subgroups are not normal ($A_p$ contains
$(p-2)!$ Sylow $p$-subgroups).
We have
\begin{align}
N_G(\langle e_i\rangle)&=\big(\langle e_i,t_i\rangle\times
\Sym_{\{1,\dots,(i-1)p,ip+1,\dots n\}}\big)\cap G
\qbox{and}\label{uu1}\\
C_G(E_i)&=\big(E_i\times\Sym_{\{(i-1)p+1,\dots,ip,rp+1,\dots,n\}}\big)\cap G,\label{uu2}
\end{align}
where
$\displaystyle\langle e_i,t_i\rangle\cong C_p\rtimes C_{p-1}$\;
is the normalizer
of $\langle e_i\rangle$ in $\Sym_{\{(i-1)p+1,\dots,ip\}}$.
Now, by assumption, $r\geq 1$ and $p>3$, and, for each $1\leq i\leq r$,
we choose $c_i=((i-1)p+1~,~(i-1)p+2~,~(i-1)p+3)$. This element clearly belongs to 
$C_G(E_i)$ and a short computation shows that it does not normalize $\langle e_i\rangle$. The collection $\{c_i~,~1\leq i\leq r\}$ satisfies the hypothesis
of Theorem~\ref{theorem:generalizationcentralizernormalizer}. Indeed,
by definition $[c_i,c_j]=1$ for all $i,j$ because the $c_i$ are all
disjoint cycles.

\end{proof}

\begin{remark}
Computational evidence using GAP \cite{GAP} shows that there is no
admissible collection for $p=3$ for the alternating and symmetric
groups of degree $n$ unless $n=4,5$ or $8$. For $n=4,5$, these groups
have $3$-rank $1$ and no non-trivial normal $3$-subgroup. Then it is
immediate that there exists an admissible collection. For $A_8$, we
can take 
$E=\langle e_1,e_2\rangle$ with $e_1=(1, 2, 3)$ and
$e_2=(4,5,6)$, and put $c_1=(1,7)(2,3)$ and $c_2=(4,8)(5,6)$. 
Then, a direct check shows that $\{E_i,c_j\}$ is an admissible
collection for $A_8$ (and so for $\Sym_8$  too).
\end{remark}

\smallskip
We will now adapt the above argument to find admissible 
collections in classical finite groups with trivial $p$-core in
non-defining characteristic, and therefore show that they have the
Quillen dimension at $p$ property. By a {\em classical group}, we
mean $G$ a finite group of Lie type $A_n,~\ls2A_n,~B_n,~C_n,~D_n$ or
$\ls2D_n$, and we refer the reader to \cite[Table 22.1]{MT} for the different
isogeny types and the standard notation that we will use. In
particular, if $G$ is of type $A_{n-1}$ defined over a field with $q$
elements, then $G$ is a central quotient of a group $G^*$ such that
$\SL_n(q)\leq G^*\leq\GL_n(q)$, and $G=G^*/Z$ for a central subgroup
$Z\leq Z(G^*)$. If $Z=Z(G^*)$, then we write $PG^*=G^*/Z(G^*)$.
We also assume that $p$ is coprime to the characteristic of the field
of definition of $G$. 

\begin{notation}\label{nota:sln}
For any natural number $n$, we write
$\GL_n^+(q)=\GL_n(q)$, $\GL_n^-(q)=\GU_n(q)$,
$\SL_n^+(q)=\SL_n(q)$ and $\SL_n^-(q)=\SU_n(q)$ for linear, unitary,
special linear and special unitary groups respectively. So a
superscript ``$+$'' means ``linear'' and ``$-$'' means ``unitary''.
We let $d$ be the multiplicative order of
$q$ modulo $p$, that is, the smallest positive integer such that $p$
divides $q^d-1$. Finally, if $n$ is an integer, we write $(n)_p$ for
the $p$-part of $n$.
\end{notation}

For the central quotient $PG=G/Z(G)$ of a classical
group $G$, recall that $O_p(G)\leq Z(G)$ except for some small linear and unitary groups. Moreover, we have isomorphisms $Z(\GL_n^\epsilon(q))\cong C_{q-\epsilon}$ and
$Z(\SL_n^\epsilon(q))\cong C_{(n,q-\epsilon)}$, for $\epsilon=\pm1$,
and the centers of symplectic and orthogonal groups are
$2$-groups (cf. \cite[Table 24.2]{MT}). Hence,
$PG$ is a $p'$-quotient of $G$ if and only if
$O_p(G)=1$. If so, $PG$ and $G$ have the same
$p$-subgroup structure \cite[Lemma 0.11]{AS1993}, in particular,
$\A_p(G)\cong \A_p(PG)$. In this case, we show that
$\Q\D_p$ holds for both groups by exhibiting an admissible 
collection for $G$, and then applying Remark
\ref{rem:faithfultoquotient}, we obtain that $PG$ possesses an
admissible  collection too. 

In the remaining case, i.e., if $O_p(G)>1$, we show that there is no
admissible collection for $G$, using Lemma \ref{lem:pstable}, and we
prove that $\Q\D_p$ holds for $PG$ by exhibiting an admissible
collection. The next table summarizes the different cases, indicating
the conditions and the results to be applied in each case. For
conciseness, we omitted from the table the exceptions $(p,q)=(3,2)$,
which apply whenever $G$ is not a linear group, and the cases $\PSL_3(q)$ with $q\equiv 1\pmod{3}$ at the prime $3$.

\begin{center}
  \begin{tabular}{ | l | c | c | c |}
    \hline
    Group $G$ & \multicolumn{2}{c|}{$O_p(G)=1$} & $O_p(G)>1$ \\ \hline
    &\multicolumn{2}{c|}{$d>1$}&$d=1$\\
    $\GL_n(q)$&\multicolumn{2}{c|}{$\Q\D_p$ by \ref{prop:sln-d-gt-1},
      \ref{thm:classic-dpos}}&$\nexists$ admissible coll. by \ref{lem:sln-p-stable}\\
    $\PGL_n(q)$&\multicolumn{2}{c|}{$\Q\D_p$ by \ref{rem:faithfultoquotient}}&$\Q\D_p$ by \ref{cor:psln}\\\hline
   &$d>1$&$d=1,(p,n)=1$&$d=1,(p,n)=p$\\
   $\SL_n(q)$&$\Q\D_p$ by \ref{prop:sln-d-gt-1}, \ref{thm:classic-dpos}& $\Q\D_p$ by \ref{prop:sln-d-eq-1}& $\nexists$ admissible coll. by \ref{lem:sln-p-stable}\\
   $\PSL_n(q)$&$\Q\D_p$  by \ref{rem:faithfultoquotient}&$\Q\D_p$ by \ref{rem:faithfultoquotient}&$\Q\D_p$ by \ref{cor:psln}\\\hline    
    &\multicolumn{2}{c|}{$d\neq 2$}&$d=2$\\
 $\GU_n(q)$&\multicolumn{2}{c|}{$\Q\D_p$ by \ref{thm:classic-dpos}}&$\nexists$ admissible coll. by \ref{lem:sln-p-stable}\\
 $\PGU_n(q)$&\multicolumn{2}{c|}{$\Q\D_p$ by \ref{rem:faithfultoquotient}}&\\\hline
&\multicolumn{2}{c|}{$d\neq 2$}&$d=2,(p,n)=p$\\
 $\SU_n(q)$&\multicolumn{2}{c|}{$\Q\D_p$ by \ref{thm:classic-dpos}}&$\nexists$ admissible coll. by \ref{lem:sln-p-stable}\\
 $\PSU_n(q)$&\multicolumn{2}{c|}{$\Q\D_p$ by \ref{rem:faithfultoquotient}}&\\\hline 
 symplectic&\multicolumn{2}{c|}{$\Q\D_p$ for $G$ by  \ref{thm:classic-dpos}}& \\
orthogonal &\multicolumn{2}{c|}{$\Q\D_p$ for $PG$ by \ref{rem:faithfultoquotient}}&--\\
    \hline
  \end{tabular}
\end{center}

\smallskip

Note that $d=1$ if and only if $q\equiv1\pmod p$, and that $d=2$ if
and only if $q\equiv-1\pmod p$, which determines when a linear or
unitary group has a nontrivial $p$-core. The last two rows denote any
symplectic or orthogonal finite group of Lie type, as
described in \cite[Table 22.1]{MT}.

\begin{prop}\label{lem:sln-p-stable}
Suppose that $p$ is odd and that $n\geq 1$.
Let $G=\GL_n^\epsilon(q)$, or
$G=\SL_n^\epsilon(q)$, and assume that $q-\epsilon\equiv0\pmod p$, or
$(q-\epsilon,n)\equiv0\pmod p$, respectively.
Then, there exists no admissible collection for $G$. 
\end{prop}

\begin{proof}
Recall that $Z(\GL_n^\epsilon(q))$ is cyclic of order
$q-\epsilon$ and $Z(\SL_n^\epsilon(q))$ is cyclic of order
$(q-\epsilon,n)$. So, the assumptions say that $O_p(G)\cap Z(G)>1$ and
the assertion follows from Lemma~\ref{lem:pstable}.
\end{proof}

We now consider the linear groups with $d>1$.
 
\begin{notation}\label{nota:betas}
We denote by $\beta_{ab}$ the $n\times n$ matrix with coefficients in
$\FF_q$ satisfying $(\beta_{ab})_{cd}=0$ unless $a=c$ and $b=d$ , in
which case $(\beta_{ab})_{ab}=1$. 
\end{notation}

\begin{prop}\label{prop:sln-d-gt-1}
Suppose that $p$ is odd and let $G=\GL_n(q)$ or $G=\SL_n(q)$
with $d>1$ except $\SL_2(2)$ for $p=3$. Then there exists an admissible collection for
$G$.
\end{prop}

\begin{proof}
Write $n=rd+f$, with $0\leq f<d$. Then the unique (up to
$G$-conjugacy) maximal elementary abelian $p$-subgroup $E$ of $G$ has
rank $r$.
We split the proof into two cases. First, suppose that $(p,q)\neq(3,2)$.
Using \cite[Section~4.10, 4.8]{GLSIII}, and the fact that
$\GL_1(q^d)\hookrightarrow\GL_d(q)$, we can choose
$E\leq G$ as the set of block diagonal matrices of the form
$$E=\langle\theta_1(u),\dots,\theta_r(u)\rangle
\qbox{where $u\in\GL_1(q^d)$ has order $p$, and where}$$
$$\begin{array}{rcccl}
\theta_i~:~\GL_1(q^d)&\to&\GL_d(q)&\to&\GL_n(q)\\
x&\mapsto&X&\mapsto&\diag(I_{(i-1)d},X,I_{n-id})\end{array}.$$
Put $e_j=\theta_i(u)$ for $1\leq i\leq r$.
Note that $\det(e_i)=1$ because $\det(e_i)$ must be a $p$-th root of
$1$ in $\FF_q$ and $(p,q-1)=1$.
Then 
\[
C_G(E)=G\cap\bigg(\diag\big(\im(\theta_1),\dots,\im(\theta_r)\big)
\times\GL_f(q)\bigg)
\]
with $\im(\theta_i)\cong\GL_1(q^d)$ for all $1\leq i\leq r$.
We have
$$C_G(E_i)\cong G\cap\big(C_{\GL_n(q)}(E)\GL_{d+f}(q)\big)\qbox{and}$$
$$N_G(\langle e_i\rangle)\cong
G\cap\big((\GL_1(q^d)\rtimes C_d)\times\GL_{n-d}(q)
\big)$$
where the factors $\GL_{d+f}(q)$ and $\GL_{n-d}(q)$
intersect in a subgroup isomorphic to $\GL_f(q)$. 
So we can choose
$c_i\in C_G(E_i)\setminus N_G(\langle e_i\rangle)$ in the factor
$\GL_{d}(q)$ properly containing the subgroup
$(\GL_1(q^d)\rtimes C_d)$ in the $i$-th diagonal block for all $i$.
Such choice is possible whenever $p$ is odd, unless  $q=2$, $d=2$ and
$p=3$, because $\GL_1(2^2)\rtimes C_2\cong\GL_2(2)\cong \Sym_3$. Now,
our choice satisfies
$c_i\in C_G(E_i)\setminus N_G(\langle e_i\rangle)$ and $[c_j,c_i]=1$
for all $i,j$. Thus $\{E_i,c_j\}$ is an admissible collection for $G$. 
 
It remains to handle the case $(p,q)=(3,2)$. If $n=2$, then
$O_3(\SL_2(2))\neq 1$ and this case is excluded by assumption. If
$n=3$, recall that a Sylow $3$-subgroup of $G=\SL_3(2)$ has order $3$
and it is not normal in $G$, saying that, trivially, there exists an
admissible collection for $G$. If
$n\geq 4$, we observe that it is enough to consider the groups
$G=\SL_{2k}(2)$, for $k\geq 2$, as the index of $\SL_{2k}(2)$ in
$\SL_{2k+1}(2)$ is not divisible by $3$, so that an admissible
collection for $\SL_{2k}(2)$ induces an admissible collection for
$\SL_{2k+1}(2)$. As above, we can choose an
elementary abelian $3$-subgroup $E$ of $G$ of maximal rank embedded in
the subgroup formed by the diagonal $2\times2$ blocks in
$\SL_{2k}(2)$.
Moreover, we see that it suffices to find admissible collections with
respect to such an $E$ for $G=\SL_4(2)$ and for $G=\SL_6(2)$, because
from them we can obtain admissible collections for $\SL_{2n}(2)$, for
all $n\geq2$.
Using GAP \cite{GAP}, we find the following admissible
collections. Let
$X=\bigl(\begin{smallmatrix}0&1\\1&1\end{smallmatrix}\bigr)$
and
$Y=\bigl(\begin{smallmatrix}0&1\\1&0\end{smallmatrix}\bigr)$. If
$G=\SL_4(2)$, put
$$e_1=\begin{pmatrix}X^2\\&X^2\end{pmatrix},\quad
e_2=\begin{pmatrix}X^2\\&X\end{pmatrix},\quad
c_1=\begin{pmatrix}X\\Y&X\end{pmatrix}\;\hbox{and}\;
c_2=\begin{pmatrix}X\\X&X^2\end{pmatrix},$$
where a blank entry means the zero matrix of appropriate size.
A direct calculation shows that $E=\langle e_1,e_2\rangle=C_3\times C_3$ is
an elementary abelian $3$-subgroup of $G$ of maximal rank,
with $c_j\in C_G(e_{3-j})$ and $c_j\notin N_G(\langle e_j\rangle)$ for
$j=1,2$, and also that $[c_1,c_2]=1$. Thus these elements form an
admissible collection for $\SL_4(2)$.
If $G=\SL_6(2)$, put
$$e_1=\begin{pmatrix}X^2\\&I_2\\&&I_2\end{pmatrix},\quad
e_2=\begin{pmatrix}X\\&X\\&&X\end{pmatrix}\;\hbox{and}\;
e_3=\begin{pmatrix}X^2\\&X^2\\&&X\end{pmatrix},$$
and
$$c_1=\begin{pmatrix}X\\X&I_2\\&&I_2\end{pmatrix},\;
c_2=\begin{pmatrix}X\\&X&Y\\&&X\end{pmatrix}\;\hbox{and}\;
c_2=\begin{pmatrix}X\\&X&X\\&&X^2\end{pmatrix}.$$
A direct computation shows that
$E=\langle e_1,e_2,e_3\rangle=C_3\times C_3\times C_3$ is an
elementary abelian $3$-subgroup of $G$ of maximal rank, with
$c_i\in C_G(\langle e_j,e_k\rangle)$ and
$c_i\notin N_G(\langle e_i\rangle)$ for any $\{i,j,k\}=\{1,2,3\}$, and
also that $[c_i,c_j]=1$ for any $1\leq i,j\leq3$. Thus these elements form an
admissible collection for $\SL_6(2)$.
\end{proof}

The cases $G=\SL_n(q)$, $\PSL_n(q)$,
and $\PGL_n(q)$ when $d=1$ are more subtle.
In this case, $O_p(\SL_n(q))=1$ if and only if $(n,p)=1$, i.e. when
$\SL_n(q)$ has no scalar matrix $\zeta I_n$ where
$\zeta^n=\zeta^p=1\neq\zeta$.
From \cite[Section 4.10]{GLSIII},
the $p$-ranks of $\SL_n(q)$, $\PSL_n(q)$ and $\PGL_n(q)$ may be found, and we use this information in the proofs of the next two propositions.

\begin{prop}\label{prop:sln-d-eq-1}
Let $G=\SL_n(q)$ with $p$ an odd prime dividing $q-1$ and
coprime to $n$. Then there exists an admissible
collection for $G$. 
\end{prop}

\begin{proof}
By assumption $O_p(G)=1$ because $p$ does not divide $n$.
Let $u\in\GL_1(q)$ be an element of order $p$ and put,
for all $1\leq i\leq n-1$,
$$e_i=\diag(uI_{i-1},u^{1-n},uI_{n-i}).$$

Let $E=\langle e_1,\dots,e_{n-1}\rangle$, so that $E\in\A_p(G)$ has
maximal rank $n-1=\rk(G)$. Then
$C_G(E)=T=C_G(T)\cong\GL_1(q)^{n-1}$ is the subgroup of $G$ formed by
all diagonal matrices with determinant $1$. We have
$$C_G(E_i)\cong G\cap\big(T\GL_2(q)\big)\qbox{and}
N_G(\langle e_i\rangle)=C_G(e_i)\cong\GL_{n-1}(q),$$
where such $\GL_{n-1}(q)$ contains $T$, and where
$$C_G(E_i)\cap N_G(\langle e_i\rangle)=T.$$
Choose $c_j=I_n+\beta_{jn}$ (see Notation \ref{nota:betas}). A routine calculation shows that
$c_j\in C_G(E_j)\setminus N_G(\langle e_j\rangle)$ has order $q$, and that
$[c_j,c_i]=1$ for all $i,j$. 
Thus $\{E_i,c_j\}$ is an admissible collection for $G$.
\end{proof}

\begin{prop}\label{cor:psln}
Let $G=\PSL_n(q)$ or $\PGL_n(q)$ with $p$ an odd prime prime dividing
$(n,q-1)$ or $q-1$ respectively, excluding the case $\PSL_3(q)$ for $p=n=3$. Then there exists an admissible  collection for $G$.
\end{prop}

\begin{proof}
First consider $G=\PSL_n(q)$. If $p=n=3$ then the $p$-rank is $2$ and we exclude this case by assumption. Otherwise, from \cite[Theorem 10.6(1)]{GL} we have
\[
\rk_p(G)=\begin{cases} n-2\text{, if $(n)_p\geq (q-1)_p$,}\\n-1\text{, if $(n)_p< (q-1)_p$.}\end{cases}
\]
For the former case, we embed $\SL_{n-1}(q)$ in $\PSL_n(q)$ as the subgroup
$$\left(\begin{array}{c|c}\SL_{n-1}(q)&
\begin{matrix}0\\\vdots\\0\end{matrix}\\\hline
0\ \dots\ 0&1\end{array}\right),$$ 
and, since $p$ does not divide $n-1$, we invoke Proposition~\ref{prop:sln-d-eq-1} to construct an
admissible  collection with the maximal elementary abelian $p$-subgroup $E$ of $G$ 
of rank $n-2$ and $c_1,\dots,c_{n-2}$ all sitting in $\SL_{n-1}(q)$.

For the case $G=\PSL_n(q)$ and $(n)_p< (q-1)_p$, let $z,u\in \FF_q^*$ be elements such that $z$ has order $(n)_p$ and $u^p=z$. Set
\[
e=\diag(z,\ldots,z)\in Z(\SL_n(q))
\]
and, as in Proposition \ref{prop:sln-d-eq-1}, set for all $1\leq i\leq n-1$,
\[
e_i=\diag(uI_{i-1},u^{1-n},uI_{n-i})\in \SL_n(q).
\]
Then $e_i^p=e$ and in the quotient $G=\SL_n(q)/Z(\SL_n(q))$ the classes of the elements $e_1,\ldots,e_{n-1}$ generate a subgroup $E\in \A_p(G)$ of rank $n-1=\rk_p(G)$. Choose $c_j$ to be the class of the element $I_n+\beta_{jn}\in \SL_n(q)$ (see Notation \ref{nota:betas}). A routine calculation shows that we have built an admissible collection for $G$.

Next let $G=\PGL_n(q)$. Then $\rk_p(G)=n-1$. We choose elements
$e_1,\dots,e_{n-1}$ as follows: Let $u\in\GL_1(q)$ be an element of order $p$ and and set
$e_i=\diag(uI_{i-1},1,uI_{n-i})$ for $1\leq i\leq n-1$.
Let $E=\langle e_1,\dots,e_{n-1}\rangle$, and define
$c_j=I_n+\beta_{jn}$ (see Notation \ref{nota:betas}).
It is clear that $c_j\in C_G(E_j)$ and that $[c_i,c_j]=1$ for all $1\leq j\leq n-1$. Moreover, a routine calculation shows that the matrices $\ls{c_j}e_j$ are not diagonal, and hence $c_j\notin N_G(\langle e_j\rangle)$
and $\{E_i,c_j\}$ is an admissible collection for $G$.
\end{proof}

\smallskip

Machine computations indicate that the cases excluded in Proposition
\ref{cor:psln} are probably genuine exceptions to the existence of
admissible collections.

We are now ready to handle the remaining classical groups (cf. the
paragraph above Notation~\ref{nota:sln} for our conventions). We follow 
the methods in \cite[Section 8]{GL}, and so let $G$ act on a vector
space $V$, defined over $\FF_q$, unless $G$ is unitary, in which case
$V$ is defined over $\FF_{q^2}$, and $V$  comes equipped with a
hermitian form. Let $\Isom(V)$ denote the full isometry group of
$V$. So, for instance, if $G$ has Lie type $A_{n-1}(q)$, then
$\Isom(V)=\GL_n(q)$, and $G$ is a central quotient of a group $G^*$
such that $\SL_n(q)\leq G^*\leq\GL_n(q)$.

\begin{thm}\label{thm:classic-dpos}
Let $G$ be a finite linear, unitary, symplectic or orthogonal group
defined over a field with $q$ elements (respectively $q^2$ elements if
$G$ is unitary), and let $p$ be an odd prime with $p\nmid q$.
Suppose that the following conditions hold:
\begin{itemize}\item[(i)] $p$ divides $|G|$ and $O_p(G)=1$.
\item[(ii)] $(p,q)\neq(3,2)$ unless $G$ is linear.
\item[(iii)] $p\nmid q+1$ if $G$ is unitary.
\end{itemize}
Then there exists an admissible 
collection for $G$.   
\end{thm}

\begin{remark}\label{except-iso2}
Some well known properties of small finite groups of Lie type
(cf. \cite[Remark 24.18]{MT} for instance), and
computational evidence using GAP \cite{GAP} show that there are no
admissible collections for small dimensional unitary, symplectic and
orthogonal groups for $p=3$ and $q=2$. 
So condition (ii) in the statement is necessary.
The case when $G$ is linear and $(p,q)=(3,2)$ is dealt
with in Proposition \ref{prop:sln-d-gt-1}. 
\end{remark}

\begin{proof}
From Propositions~\ref{prop:sln-d-eq-1} and~\ref{cor:psln}, we can
assume that if $G$ is linear, then $d>1$.
Condition (iii) is equivalent to saying that if $G$ is unitary, then
$d\neq2$. 

Let $V$ be the underlying vector space of $G$, and $\Isom(V)$ its full
isometry group, as explained above. From
\cite[Tables 22.1 and 24.2]{MT}, our assumptions
allow us to suppose that $G$ is simply connected.
Indeed, the index of the group of simply
connected type in $\Isom(V)$ is coprime to $p$, so that both have have
isomorphic Sylow $p$-subgroups. Thus, if there is an admissible
collection for the group of simply connected type, then its inclusion
in $\Isom(V)$ yields an admissible collection for
$\Isom(V)$. Hence, if $G$ is a central quotient of $G^*$, the
image of such admissible collection is an admissible collection for
$G$, by Remark~\ref{rem:faithfultoquotient}. For instance, if $G$
has Lie type $A_{n-1}(q)$, then we can suppose that
$G=\SL_n(q)$. Indeed, if $d>1$ (i.e. if $q\not\equiv1\pmod p$), then
an admissible collection for $\SL_n(q)$ gives admissible
collections for any group $G^*/Z$, where
$\SL_n(q)\leq G^*\leq\GL_n(q)$ and $Z\leq Z(G^*)$.

Let $E$ be a maximal elementary abelian $p$-subgroup of $G$ of maximal
order. 
Choose generators $\{e_1,\dots,e_r\}$ of $E$ and a basis
$\{v_1,\dots,v_n\}$ of $V$ such that $e_i$ acts as the identity
everywhere except on $V_i=\langle v_{(i-1)d+1},\dots,v_{id}\rangle$
for all $1\leq i\leq r$. Put $V_0=\langle v_{rd+1},\dots,v_n\rangle$,
possibly $V_0=\{0\}$. Thus $V=\oplus_{0\leq i\leq r}V_i$. We refer the
reader to \cite[Table 10:1]{GL} for the values of $r$
depending on the type of $G$, and the values of $n$ and $d$. (Recall
that under our assumptions, $G$ and $\Isom(V)$ have isomorphic Sylow
$p$-subgroups.) 

From \cite[8-1]{GL}, we have, for all $1\leq i\leq r$,
$$C_G(e_i)=G\cap\big(\Isom(V/V_i)\times\GL_1^\epsilon(q^e)\big).$$
It follows that
\begin{multline*}
  C_G(E_i)=\bigcap_{j\neq i}C_G(e_j)=\\
  G\cap\bigg(\GL_1^\epsilon(q^e)^{(1)}\times\cdots\times\GL_1^\epsilon(q^e)^{(i-1)}\times
  \Isom(V_i\oplus V_0)\times\GL_1^\epsilon(q^e)^{(i+1)}\times\cdots\times\GL_1^\epsilon(q^e)^{(r)}\bigg)
\end{multline*}
for $1\leq i\leq r$, and
$$C_G(E)=G\cap\big(\GL_1^\epsilon(q^e)^{(1)}\times\cdots\times\GL_1^\epsilon(q^e)^{(r)}\times\Isom(V_0)\big),$$
where $\GL_1^\epsilon(q^e)^{(i)}=C_{\Isom(V_i)}({e_i}_{|_{V_i}})$ is
the centralizer of the action of $e_i$ restricted to $V_i$, where the
sign $\epsilon=\pm1$ and the value of $e$ depend on the parity of $d$
and the type of $G$, and where $\Isom(V_0)$ has the same type as $G$.
Explicitly,
\begin{itemize}\item
If $G$ is linear: $\epsilon=+$ and $d=e$.
\item If $G$ is unitary: if $d\equiv2\pmod4$, then $\epsilon=-$, and
$\epsilon=+$ otherwise. We put $e=2d$ if $d$ is odd, $e=\frac d2$ if
$d\equiv1\pmod4$ and $e=d$ if $d\equiv0\pmod4$.
\item If $G$ is symplectic or orthogonal, put $f=\frac12\lcm(2,d)$ and hence
$\epsilon=+$ if $f$ is odd and $-$ otherwise, and put $e=2f$. (We
refer the reader to \cite[Section 8]{GL} for the type of the
orthogonal space, as it does not impact on our argument.)
\end{itemize}

By definition, $\dim(V_0)<d$ (resp $2d$ if $G$ is symplectic or
orthogonal), saying that $p\nmid|\Isom(V_0)|$.

Also,
$$N_G(\langle e_i\rangle)=G\cap\bigg(\big(\GL_1^\epsilon(q^e)^{(i)}\rtimes
C_e^{(i)}\big)\times\Isom(\oplus_{j\neq i}V_j)\bigg),$$
where $e$ is as above.

Conditions (i)-(iii) in the statement ensure that
$\big(\GL_1^\epsilon(q^d)^{(j)}\rtimes C_e^{(j)}\big)\lneq\Isom(V_j)$,
so that we can pick $c_j\in G$ such that $c_j$ acts as the identity on
$V_i$ for all $i\neq j$ and the restriction
${c_j}_{|V_j}\in\Isom(V_j)\setminus\big(\GL_1^\epsilon(q^d)^{(j)}\rtimes
C_e^{(j)}\big)$. 

Since $V=\oplus_{0\leq i\leq r}V_i$, we have
$[\Isom(V_i),\Isom(V_j)]=1$ for all $i\neq j$, and therefore the
elements $c_j$ commute pairwise, i.e.
$[c_i,c_j]=1$ for any $1\leq i,j\leq r$. We conclude that
$\{E_i,c_j\}$ is an admissible collection for $G$.
\end{proof}

\smallskip

Our objective in this work was to devise a simpler argument to show
the Quillen dimension at $p$ property for the symmetric and
alternating groups, and for the finite classical groups
in non-defining characteristic. Although our results do not fully meet
our initial objective, we believe that our methods can be further
generalized to tackle the cases excluded from the present paper, and
also $p$-extensions, i.e., almost-simple groups with an elementary
abelian $p$-group inducing outer automorphisms. The first author will
pursue this aim in a subsequent work.  

\medskip
\noindent{\bf Acknowledgements.}\;
The authors are sincerely grateful to Gunter Malle for his helpful suggestions in this research work.

\end{document}